\documentclass{amsart}
\usepackage{amssymb}
\usepackage{amsmath,amsfonts,amsthm,amstext}
\usepackage[all]{xy}
\newtheorem{theorem}{Theorem}[section]
\newtheorem{lemma}[theorem]{Lemma}
\newtheorem{prop}[theorem]{Proposition}
\newtheorem{proposition}[theorem]{Proposition}
\newtheorem{cor}[theorem]{Corollary}

\newtheorem{letterthm}{Theorem}

\theoremstyle{definition}

\newtheorem{remark}[theorem]{Remark}

\numberwithin{equation}{section}

\newcommand{\Z}{\mathbb{Z}}

\newcommand{\Q}{\mathbb{Q}}
\newcommand{\QQ}{\mathbb{Q}Q}

\newcommand{\ra}{\rangle}
\newcommand{\vb}{{\rm{vb}_1}}
\newcommand{\vbp}{{\rm{vb}_1^{(p)}}}
\newcommand{\FP}{{\rm{FP}}}
\newcommand{\Hom}{{\rm{Hom }}}
\newcommand{\Tor}{{\rm{Tor }}}
\newcommand{\ann}{{\rm{ann}}}
\newcommand{\ns}{\vartriangleleft}

\newcommand{\into}{\rightarrowtail} 
\newcommand{\onto}{\twoheadrightarrow}

\begin{document}

\title{The virtual first  Betti  number of  soluble groups}

\author{Martin R. Bridson}

\address{Mathematical Institute, 
University of Oxford, Andrew Wiles Building, Woodstock Road,
Oxford OX2 6GG,
UK} 
\email{bridson@maths.ox.ax.uk}

\author{Dessislava H. Kochloukova}

\address
{Department of Mathematics, University of Campinas (UNICAMP), 
13083-859 Campinas, SP, Brazil} 
\email{desi@ime.unicamp.br}



\begin{abstract} We show that if a group $G$ is  finitely presented  and nilpotent-by-abelian-by-finite, then there is an upper bound on $\dim_{\mathbb{Q}} H_1(M, \mathbb{Q})$, where $M$ runs through all subgroups of finite index in $G$.
\end{abstract}

\maketitle

\section{Introduction} 
The {\em virtual first  betti number} of a finitely generated group $G$ is defined as
$$
\vb(G) = \sup \{ \dim H_1(S,\Q) \mid S\le G\text{ of finite index }\}.
$$
A group is said to be {\em large} if it has a subgroup of finite index that maps onto a non-abelian free
group.  If $G$ is large then $\vb(G)=\infty$. It is easy to find finitely generated groups $G$ that are not large
but have $\vb(G)=\infty$. For example, in the metabelian group
 $\Z\wr\Z=\<a,t\mid \forall n, [a,t^{-n}at^n]=1\ra$, the subgroup $S_m<\Z\wr\Z$ generated by $t^m$ and the
 conjugates of $a$ 
has index $m$ and $H_1(S_m,\Z)=\Z^{m+1}$. In contrast, no example is known of a {\em finitely presented} group that is not large
but has  $\vb(G)=\infty$ (cf.~\cite{button}, \cite{lack}).
Since amenable groups do not contain non-abelian free subgroups, one might hope to
resolve this issue by finding a finitely presented amenable group with  $\vb(G)=\infty$, but
this seems to be a non-trivial matter. 

We shall prove in this paper that for large classes of finitely presented
soluble groups  $\vb(G)$ is always finite. One would like to prove that the same is true for all finitely presented soluble groups,
but here one faces the profound difficulty of deciding which soluble groups
admit finite presentations; this is unknown even for  abelian-by-polycyclic and nilpotent-by-abelian groups. 

In the case of metabelian groups, finite
presentability is completely understood in terms of the Bieri-Strebel invariant 
\cite{BieriStrebel}.  
 Some sufficient conditions for finite presentability of  nilpotent-by-abelian groups were considered by Isaac \cite{Isaac} and later Groves \cite{Groves}. In the case of $S$-arithmetic nilpotent-by-abelian groups $G$ one knows
 more thanks to the work of Abels \cite{Abels}: if $G$ is an extension of a nilpotent group $N$ by an abelian group $Q$ then $G$ is finitely presented if and only if 
 it is of type $\FP_2$, which it is if and only if $H_2(N, \mathbb{Z})$ is finitely generated as a $\mathbb{Z} Q$-module (where the $Q$ action is induced by conjugation) and\footnote{throughout this article, $H'$ denotes the commutator subgroup of a group $H$} $G/ N'$ is finitely presented as a group. The first
 of these conditions is an easy consequence of the fact that $\mathbb{Z} Q$ is a Noetherian ring and the second is a corollary of a result of Bieri and Strebel that every metabelian quotient of a group of type $\FP_2$ that does not contain non-cyclic free subgroups is finitely presented \cite{BieriStrebel}. 
The case where $G$ is an extension of an abelian  normal subgroup $A$ by a polycyclic group $Q$ was approached by Brookes and Groves who studied modules  over crossed products of a division ring by a free abelian group \cite{ChrisJohn1}, \cite{ChrisJohn2} and \cite{ChrisJohn3}.  

\medskip 
Given this background, the natural place to begin our investigation into the   virtual first betti number of finitely presented soluble groups
is in the setting of metabelian groups.  Using methods from commutative algebra,  we prove (Theorem \ref{metabeliancase}) that if $G$ is finitely presented and
metabelian, then $\vb(G)$ is finite. (The hypothesis that one actually needs to impose on $G$ is somewhat weaker than finite presentability; see Remark \ref{r:weak-fp}.)
The metabelian case is used in the proof of our main theorem, which is the following.  

\begin{letterthm}\label{n-by-a-by-f}
 Let $G$ be a finitely presented group. If $G$ is nilpotent-by-abelian-by-finite, then $\vb(G)$ is finite.
\end{letterthm}

Our proof of this theorem
relies on the fact that all metabelian quotients  of soluble groups of type $\FP_2$ are finitely presented \cite[Thm.~5.5]{BieriStrebel}, as well as a technical
result concerning the homology of subgroups of finite index (Proposition \ref{nilpotent}). Groves, Kochloukova and Rodrigues 
 \cite[Thm.~A]{GKR} proved that if an abelian-by-polycyclic group $G$ is of type $\FP_3$ then it is nilpotent-by-abelian-by-finite, in which case $\vb(G)$ is finite
by Theorem \ref{n-by-a-by-f}. The same is true of all
soluble groups of type $\FP_{\infty}$, because they  are constructible \cite{constructible}, hence nilpotent-by-abelian-by-finite, but in this case stronger finiteness results were already known:  constructible soluble groups are obtained from the trivial group by finite sequences of ascending HNN extensions and finite extensions,
from which it follows that they have finite Pr\"ufer rank (i.e.~there is an upper bound on the number of generators for the finitely generated subgroups).

\smallskip

It is natural to wonder if
Theorem A might remain true when the field of rationals $\mathbb{Q}$ in the definition of virtual betti number is replaced with
other coefficient fields, such as the field with $p$ elements $\mathbb{F}_p$. We shall see in Section 5 that it does not.

\medskip
\noindent{\bf Conjecture:} {\em  If $G$ is finitely presented and soluble, then $\vb(G)$ is finite.}
\medskip

It is difficult to construct finitely presented soluble groups  that are not nilpotent-by-abelian-by-finite. The examples
provided by the constructions of Robinson and Strebel  \cite{RS} all satisfy the conjecture.

While editing the final version of this work, we learnt  that Andrei Jaikin-Zapirain has, in unpublished work, also proved
Theorem A in the metabelian case.
Higher dimensional analogues of Theorem A are considered in the forthcoming PhD thesis
of Mokari \cite{Fatemeh}.

\smallskip

{\bf Acknowledgements} 
We thank Fatemeh Mokari and an anonymous referee for helpful comments. We particularly
thank the referee for simplifying the proofs of Proposition 6.2 and Claim 2 in section \ref{commutativealgebra234}.  The work of the first author was supported by grants from the EPSRC
and by a Wolfson Merit Award from the Royal Society; the work of  the 
second author was supported by ``bolsa de produtividade em pesquisa", CNPq, Brazil: we thank all of these organizations.

\section{Preliminary results }

\subsection{Preliminaries  on finitely presented metabelian groups} \label{ss2.1}

We fix a short exact sequence of groups $A \into G \onto Q$,
where $A$ and $Q$ are abelian and $G$ is finitely
generated. The action of $G$ on $A$ by conjugation induces an action of $Q$, which enables us to regard
$A$ as a right $\mathbb{Z} Q$-module. Because $G$ is finitely generated and $Q$ is finitely presented, $A$ is finitely generated as a $\mathbb{Z} Q$-module.

Associated to a non-zero real character $\chi :
Q \to \mathbb{R}$ one has the monoid $$Q_{\chi} = \{ g \in Q
\mid  \chi(g) \geq 0 \}.$$ The character sphere $S(Q)$ is the set
of equivalence classes in $\Hom(Q, \mathbb{R}) \smallsetminus \{ 0
\}$ under the relation that identifies
$\chi_1 \sim \chi_2$ if $\chi_1 = \lambda \chi_2$ for some
$\lambda > 0$. We write $[\chi]$ for
the class of $\chi$.  Following Bieri and Strebel \cite{BieriStrebel}, let
$$\Sigma_A(Q) = \{ [\chi] \mid A \hbox{ is finitely generated
  as a } \mathbb{Z} Q_{\chi}\hbox{-module} \}.
$$
By definition, the $\mathbb{Z} Q$-module $A$ is {\em 2-tame} if $\Sigma_A(Q)^c=S(Q) \smallsetminus
\Sigma_A(Q)$ contains no pair of antipodal points. According to
\cite[Thm.~5.4]{BieriStrebel}, $G$ is finitely presented if and only
if $A$ is a 2-tame  $\mathbb{Z} Q$-module, and this happens
precisely when $G$ is of homological type $\FP_2$. We refer the reader to \cite{BieriQMWbook} for general results concerning  groups of type $\FP_m$.
If $A_1\into A_2\onto A_3$ is an exact sequence of $\mathbb{Z} Q$-modules, then $\Sigma_{A_2}(Q)^c = \Sigma_{A_1}(Q)^c \cup \Sigma_{A_3}(Q)^c$
(see  \cite[Prop.~2.2]{BieriStrebel}),  hence   every quotient of a
2-tame $\mathbb{Z} Q$-module is 2-tame.

\subsection{Tensor products and finite presentability} \label{s:tensor}
Let $R$ be a  noetherian commutative ring with unit $1$ and
let  $W$  be a finitely generated $R Q$-module. As above, we have a Sigma invariant $\Sigma_W(Q) = \{ [\chi]  \mid
W \hbox{ is finitely generated as }R Q_{\chi}-\hbox{module} \}$, and
$W$ is defined to be 2-tame as an $R Q$-module if $\Sigma_W^c(Q) = S(Q) \setminus \Sigma_W(Q)$ has no pair of antipodal points.

The question of when the tensor square $W \otimes_R W$ is finitely generated  as an $R Q$-module (with $Q$ acting diagonally) is addressed in \cite{BG}, where
it is shown that $[\chi]$ lies in $\Sigma_W^c(Q)$ if and only if the ring $S = R Q / \ann_{R Q} (W)$ admits 
a real valuation $v : S \to \mathbb{R} \cup \{ \infty \}$  (in the sense of Bourbaki) that extends $\chi$ and is such that the restriction $v_0$ of $v$ to
 the image $\overline{R}$ of $R$ in $S$ is non-negative and discrete. 
By \cite{BG},  $W \otimes_R W$ is finitely generated  as an $R Q$-module if and only if there is no pair of
antipodal elements $[\chi], - [\chi] \in \Sigma_W^c(Q)$ that can be lifted to valuations of $S$  that have the same restriction $v_0$ to $\overline{R}$, 
with $v_0$ discrete and non-negative. (These last conditions on $v_0$ are automatic if $\overline{R}$ is  $\mathbb{Z}$.)

Returning to the context of paragraph (\ref{ss2.1}), we
apply these general considerations with  $W = A \otimes \mathbb{Q}$ and $R = \mathbb{Q}$, in which case
 $W \otimes_R W \cong (A \otimes_{\mathbb{Z}} A) \otimes_{\mathbb{Z}} \mathbb{Q}$.
We deduce that if there exists a group extension $A \into G \onto Q$, with 
$G$ finitely presented,  then 
$W= A \otimes \mathbb{Q}$ is 2-tame as a $\mathbb{Q} Q$-module, 
and $W \otimes_R W \cong (A \otimes_{\mathbb{Z}} A) \otimes_{\mathbb{Z}} \mathbb{Q}$
 is finitely generated as a $\mathbb{Q} Q$-module via the diagonal $Q$-action.  

We shall also need a refinement of this observation that involves the annihilator $\ann_{\Z Q}(A)$ of $A$ in $\Z Q$,
which we denote $I$. In \cite[(1.3)]{BS2} Bieri and Strebel prove that 
$$
\Sigma_A(Q) = \Sigma_{\mathbb{Z} Q/ I} (Q).
$$
Thus if $A$ is
2-tame as a $\mathbb{Z} Q$-module, then so is  $\mathbb{Z} Q / I$.

\begin{lemma}\label{l:tensor} If there exists a group extension $A \into G \onto Q$ with $A$ and $Q$ abelian and
$G$ finitely presented,  and $I=\ann_{\Z Q}(A)$, then 
$(\Z Q/I) \otimes_{\mathbb{Z}} (\Z Q/I)  \otimes_{\mathbb{Z}} \mathbb{Q}$
 is finitely generated as a $\mathbb{Q} Q$-module via the diagonal $Q$-action. 
\end{lemma}

\subsection{Preliminaries on commutative algebra}
We will need the following basic facts from commutative algebra; for details see, for example,  \cite{algebrabook}, \cite{algebrabook2} or \cite{eisenbud}. 
Let $Q$ be a finitely generated abelian group and recall that the {\em Krull dimension}
of a commutative ring is the supremum of the lengths of  all chains of prime ideals in the ring.
\begin{enumerate}\label{comm}
\item The radical $\sqrt{J}$ of each ideal
$J\ns\mathbb{Q} Q$  is the intersection of the finitely many prime ideals that contain $J$ and are minimal subject to this condition. 
\item 
Finite dimensional $\mathbb{Q}$-algebras are Artinian and therefore have Krull dimension 0.
\end{enumerate}
Throughout, if $R$ is a commutative ring and $m$ a positive integer, then $R^m$ will denote the subring generated by $m$-th powers, {\bf except}
that $\Z^n$ and $\Q^n$ will denote Cartesian powers. Where no ring is specified, tensor products are assumed to be taken over $\Z$.

\section{A finiteness result in commutative algebra} \label{commutativealgebra234}

Lemma \ref{l:tensor} assures us that the following theorem applies to the modules that arise from short exact sequences 
$N \into G \onto \Z^n$ associated to finitely presented metabelian groups.

\begin{theorem} \label{commutativealgebra} 
Let $Q\cong\mathbb{Z}^n$ be a group and let $S =\mathbb{Z} Q / I$  be a commutative ring such that $(S \otimes_{\mathbb{Z}} S) \otimes_{\mathbb{Z}} \mathbb{Q}$ is finitely generated as a $\mathbb{Q} Q$-module via the diagonal $Q$-action.  Then,
$$\sup_{m}\dim_{\mathbb{Q}}  (S \otimes_{\mathbb{Z} Q^m} \mathbb{Q}) < \infty.$$
\end{theorem}

\begin{proof} 
Let $B = S \otimes \mathbb{Q} = \mathbb{Q} Q / J$ and for each positive integer $m$ define $J_m\ns \QQ$
to be $(J, Q^m - 1)$ and
$$B_m:= B \otimes_{\mathbb{Q} Q^m} \mathbb{Q}= \mathbb{Q} Q / J_m  \cong S \otimes_{\mathbb{Z} Q^m}  \mathbb{Q}.$$
As $\mathbb{Q} Q / (Q^m - 1)$ is finite
dimensional over $\mathbb{Q}$, so  is  $B_m = \mathbb{Q} Q /
J_m.$ Hence
$B_m$ has Krull dimension $0$, i.e.~every prime ideal in $B_m$ is a maximal one. Therefore,
the finite collection of primes ideals $P_{m,t}$ whose intersection is $\sqrt{B_m}$ are the only prime ideals in $\mathbb{Q} Q$ above $J_m$, and each 
of the quotients $\mathbb{Q} Q / P_{m,t}$ is a field.

We shall establish the theorem by proving the following:

\medskip
\noindent{\bf Claim 1.} {\em There exist only finitely many fields $F$ such that for some $m \geq 1$ (depending on $F$) the field $F$ is a quotient of $B_m$. }

\medskip
Claim 1 provides an integer $m_0$ such that if a field $F$ is a quotient of $B_m$
then the natural map $\mathbb{Q} Q \to F$ factors through $\mathbb{Q} Q / (Q^{m_0} - 1)$.

\medskip
\noindent{\bf Claim 2.} {\em If $m_0$ divides $m$ then $J_m = J_{m r}$  for every $r \in \mathbb{N}$.}

\medskip
To see that the theorem follows from these claims, note that for an arbitrary positive integer $m$ we have
$J_m \supseteq J_{m m_0} = J_{m_0}$, whence 
$$\dim_{\mathbb{Q}} (\mathbb{Q} Q / J_m) \leq \dim_{\mathbb{Q}} (\mathbb{Q} Q / J_{m_0}) \leq \dim_{\mathbb{Q}} (\mathbb{Q} Q / (Q^{m_0} - 1)) = \dim_{\mathbb{Q}} \mathbb{Q}[Q/ Q^{m_0}] =  m_0^n.$$
\medskip

{\bf Proof of Claim 1.} Our hypothesis on $S$ implies that $B \otimes_{\mathbb Q} B$ is finitely generated as $\mathbb{Q} Q$-module via the diagonal $Q$-action,
by $d$ elements say. Let $F$ be a field quotient of $B_m$ and let
$\theta : \mathbb{Q} Q \to F$ be the canonical projection; so $Q^m - 1 \subseteq \ker(\theta)$. Then $\theta(Q)$ is a finitely generated multiplicative subgroup of $F^*$ 
that has finite exponent and $F$, being finite dimensional over $\mathbb{Q}$,  embeds in $\mathbb{C}$. Hence $\theta(Q)$ is a finite cyclic group,
generated by a root of unity,  $\epsilon$ of order $s$ say.  Thus we obtain a subgroup $H<Q$ such that $Q/ H$ is cyclic of order $s$ and $H-1 \subseteq \ker (\theta)$.
Now, $F \cong \mathbb{Q}[x]/ (f)$, where $f$ is the minimal polynomial of $\epsilon$ over $\mathbb{Q}$. And $f$ is an irreducible factor of $x^s - 1$ in $\mathbb{Q}[x]$,
whose zeroes are distinct roots of unity with order precisely $s$. Thus  
 $\dim_{\mathbb{Q}} F = \deg(f) =  \varphi(s)$, where $\varphi$ is Euler's totient function. On other hand,
$F \otimes_\Q F$ is an epimorphic image of the $\mathbb{Q} Q$-module $B \otimes_\Q B$ and the action of $Q$ on $F \otimes_\Q F$ factors 
through the action of $Q / H$, so $F \otimes_\Q F$ is generated as ${\mathbb{Q}} [Q/ H]$-module by $d$ elements. Hence
$$
\varphi(s)^2 = (\dim_{\mathbb{Q}} F)^2 = \dim_{\mathbb{Q}} (F \otimes_\Q F) \leq d \dim_{\mathbb{Q}} \mathbb{Q} [Q / H] = d s.
$$ 
An elementary calculation shows that   $\varphi(n)/\sqrt{n}\to \infty$ as $n\to\infty$, so for fixed $d$ there are only finitely many possible values of
$s$ and $\epsilon$. Let $b$ be a natural number such that the order of $\epsilon$ is at most $b$. Then the order of $\epsilon$ is a divisor of $m_0 = b!$
and $$F \hbox{ is a quotient of } \mathbb{Q} Q / (Q^{m_0} - 1).$$ 

Since  $\mathbb{Q} Q / (Q^{m_0} - 1)$ is finite dimensional over
$\mathbb{Q}$ it has Krull dimension 0, so has only finitely many prime ideals and finitely many
field quotients. This completes  the proof of Claim 1.

\medskip  

{\bf Proof of Claim 2.}  
Since $m_0$ divides $m$ we have $J_m\subseteq J_{m_0}$, so the prime ideals  containing $J_{m_0}$ also contain $J_{m}$.
On the other hand, we saw earlier that for each of the prime ideals $P_{m,i}$ containing $J_m$, the quotient $F_i:=\mathbb{Q} Q / P_{m,i}$ is a field.
By definition, $m_0$ is such that  $\mathbb{Q} Q \to F_i$ factors through $\mathbb{Q} Q / (Q^{m_0} - 1)$, and therefore $P_{m,i}$ (which
already contains $J\subset J_m$)
contains $J_{m_0}=(J, Q^{m_0} - 1)$.  
The radical of $J_m$ is the intersection of the prime ideals containing it, so
$$
\sqrt{J_{m}} =  \sqrt{J_{m_0}}.
$$

Arguing by induction on $r$, Claim 2 will follow if we can prove that for every prime number $p$ we have
$J_{m} = J_{mp}$, which is equivalent to the assertion that $q^m-1\in J_{mp}$ for all $q\in Q$. {\bf{We fix $q\in Q$.}}

From the preceding argument we have $
\sqrt{J_{m}} =  \sqrt{J_{mp}}.
$
In particular,  $Q^m - 1 \subseteq J_m \subseteq \sqrt{J_m} = \sqrt{J_{mp}}$, so there is a natural number (over which we
have no control) $s$ such that 
\begin{equation} \label{mod-mp1}
(q^m-1)^s \in J_{mp}.
\end{equation}
As $Q^{mp} - 1 \subseteq J_{m p}$, we also have
\begin{equation} \label{mod-pm2}
 q^{mp} - 1  \in \ J_{mp}.
\end{equation}
Let $g(x)$ be the greatest common divisor of $x^{pm} - 1$ and $(x^m - 1)^s$ in $\mathbb{Q}[x]$. In characteristic zero, the polynomial $x^{pm} - 1$ has no repeated roots, so neither does $g(x)$. Since $g(x)$ divides $(x^m - 1)^s$, it must actually divide  $x^m - 1$, so in fact $g(x) = x^m - 1$. 
From (\ref{mod-mp1}), (\ref{mod-pm2}) and   B\'{e}zout's Lemma, we have $g(q) \in J_{pm}$. Since $q\in Q$ is arbitrary, this implies  that $J_{mp} = J_m$.
\end{proof}

\section{The Main Theorem for Metabelian Groups} 

In this section we prove that all finitely presented metabelian groups have finite virtual first betti number.
The proof relies on the finiteness theorem proved in the previous section and two technical lemmas, the first
of which is a simple observation about commensurable groups.

\begin{lemma} \label{finiteindex} Let $G$ be a group. If $G_0<G$ is a subgroup of finite index, then
$\vb(G)=\vb(G_0)$. 
\end{lemma}

\begin{proof} By definition, 
$
\vb(G) = \sup_M  \dim H_1(M,\Q)
$
where the supremum is taken over finite index subgroups of $G$. If $M$ has finite index in $G_0$, then it also has finite index in $G$,
so $\vb(G)\ge \vb(G_0)$. Conversely, if $S$ has finite index in $G$, then $S_0=G_0\cap S$ has finite index in $G_0$, and since
it also has finite index in $S$ we have $\dim H_1(S_0,\Q) \ge \dim H_1(S,\Q) $, so $\vb(G_0)\ge \vb(G)$.
\end{proof}

\begin{lemma} \label{fnew1}
Let  $ A \into G \onto Q$ be a short exact sequence of groups with $A$ and $Q$ abelian and let $n$ be the torsion-free rank of $Q$. Then,
\begin{itemize}
\item[(a)] writing $ [G, A] = \langle \{ [g,a] = g^{-1} a^{-1} g a \ | \ g \in G, a \in A \} \rangle$, we have
$$
\dim_{\mathbb{Q}} H_1(G, \mathbb{Q}) \leq  \dim_{\mathbb{Q}} (A/[G,A] \otimes \mathbb{Q}) + n.
$$
In the split case, $G = A \rtimes Q$, we have
$
H_1(G, \mathbb{Q}) \cong (G/[G, A]) \otimes_{\mathbb{Z}} \mathbb{Q}, 
$ and
 $$
\dim_{\mathbb{Q}} H_1(G, \mathbb{Q}) =  \dim_{\mathbb{Q}} (A/[G,A]\otimes \mathbb{Q}) + n
$$

\item[(b)] If $G_m$ is a subgroup of finite index in $G$ and $Q_m$ is the image of $G_m$ in $Q$, then 
$$
\dim_{\mathbb{Q}} H_1(G_m, \mathbb{Q}) \leq \dim_{\mathbb{Q}} (A \otimes_{\mathbb{Z} Q_m} \mathbb{Q}) + n.
$$
In the split case, $G_m = (A \cap G_m) \rtimes Q_m$,  equality is attained:
$$\dim_{\mathbb{Q}} H_1(G_m, \mathbb{Q})  =  \dim_{\mathbb{Q}} (A \otimes_{\mathbb{Z} Q_m} \mathbb{Q}) + n.
$$

\item[(c)] If $G = A \rtimes Q$ and   
$\mathcal B$ denotes the set of subgroups of finite index in $Q$, then 
$$
\vb(G) = \sup_{S \in {\mathcal B}}  \dim_{\mathbb{Q}} (A \otimes_{\mathbb{Z} S} \mathbb{Q}) + n.
$$
\end{itemize}
\end{lemma}

\begin{proof} (a) 
 As $[G, A] \subseteq [G, G]$, we see that $H_1(G, \mathbb{Z}) = G/ [G, G]$ is a quotient of $G/[G, A]$. 
 So from the central extension $A / [G, A] \into G/ [G, A] \onto Q$ we get
 $$
 \dim_{\mathbb{Q}} H_1(G, \mathbb{Q}) \leq \dim_{\mathbb{Q}} (A/[G,A] \otimes \mathbb{Q}) + \dim_{\mathbb{Q}} (Q \otimes \mathbb{Q}) = 
 \dim_{\mathbb{Q}} (A/[G,A] \otimes \mathbb{Q}) + n.
 $$

If $G = A \rtimes Q$ then, using that $A, Q$ are abelian and $A$ is normal in $G$, we get $[G, G] = [A Q, A Q] = [Q,A] \subseteq [G, A] \subseteq [G,G]$, hence $[G,G] = [G,A]$ and  $A / [G, A] \into G/ [G, G] \onto Q$ is an exact sequence of abelian groups.

\medskip

For (b) we consider the short exact sequence $A_m\into G_m\onto Q_m$, where $A_m=A\cap G_m$. From part (a) we have
\begin{align}\label{start}
\dim_{\mathbb{Q}} H_1(G_m, \mathbb{Q}) \leq \dim_\Q (A_m \otimes_{\Z Q_m} \mathbb{Q}) \, +  \, n,
\end{align}
with equality if the sequence splits. Furthermore, since $A / A_m$ is finite we have
$$
0 = \Tor_1^{\mathbb{Z} Q_m}( A / A_m , \mathbb{Q})  \hbox{ and } (A / A_m) \otimes_{\mathbb{Z} Q_m} \mathbb{Q} = 0.
$$
Thus there is an exact sequence (part of the long exact sequence in $\Tor$ associated to  $A \cap G_m \into A \onto A / (A \cap G_m)$)
$$
0 = \Tor_1^{\mathbb{Z} Q_m}( A / A_m , \mathbb{Q}) \to A_m \otimes_{\mathbb{Z} Q_m} \mathbb{Q} \to A \otimes_{\mathbb{Z} Q_m} \mathbb{Q}
\to (A / A_m) \otimes_{\mathbb{Z} Q_m} \mathbb{Q} = 0,
$$
whence $A_m \otimes_{\mathbb{Z} Q_m} \mathbb{Q} \cong A \otimes_{\mathbb{Z} Q_m} \mathbb{Q}$.
Thus we may replace $A_m \otimes_{\Z Q_m} \mathbb{Q}$ in (\ref{start}) by $A \otimes_{\mathbb{Z} Q_m} \mathbb{Q}$,
and (b) is proved. 
\medskip

(c) From the first part of (b) we have 
$$\vb(G)\le \sup_{S \in {\mathcal B}}  \dim_{\mathbb{Q}} (A \otimes_{\mathbb{Z} S} \mathbb{Q}) + n,$$
and to obtain the reverse inequality we use the second part of (b) 
\begin{equation*}
\sup_{S \in {\mathcal B}}  \dim_{\mathbb{Q}} (A \otimes_{\mathbb{Z} S} \mathbb{Q}) + n = 
\sup_{S \in {\mathcal B}} \dim_{\mathbb{Q}} H_1( A \rtimes S, \mathbb{Q}),
\end{equation*}
noting that $A\rtimes S$ has finite index in $G$.
\end{proof}

\begin{theorem} \label{metabeliancase}
Let $A \into G \onto Q$ be a short exact sequence of groups with $A$ and $Q$ abelian. If $G$ is finitely presented
then its virtual first betti number $\vb(G)$ is finite. 
\end{theorem}

\begin{proof} By passing to a subgroup of finite index in $Q$ and replacing $G$ by the inverse image of this subgroup,
we may assume that $Q$ is free abelian. 
Lemma \ref{finiteindex} assures us that it is enough to consider this case, and  Lemma \ref{fnew1}(b)
tells us that we will be done if we can establish an upper bound on  $ 
\dim_{\mathbb{Q}} ( A  
\otimes_{\mathbb{Z} Q_m} \mathbb{Q})$ as $Q_m$ ranges over the subgroups of finite index in $Q$.

Recall that $A$ is finitely generated as a $\mathbb{Z} Q$-module, say by $d$ elements. Thus,
denoting the  annihilator $ \ann_{\mathbb{Z} Q}(A)  = \{ \lambda \in \mathbb{Z} Q \mid A \lambda = 0 \}$ 
by $I$,  we have an epimorphism of $\mathbb{Z} Q$-modules
$$
(\mathbb{Z} Q / I)^{[d]} = \mathbb{Z} Q / I \oplus \ldots \oplus \mathbb{Z} Q / I \to A
$$
that induces an epimorphism of $\mathbb{Q}$-vector spaces
$$
((\mathbb{Z} Q / I) \otimes_{\mathbb{Z} Q_m} \mathbb{Q}   )^{[d]} = (\mathbb{Z} Q / I)^{[d]} \otimes_{\mathbb{Z} Q_m} \mathbb{Q} \to A  
\otimes_{\mathbb{Z} Q_m} \mathbb{Q}.
$$
Thus
$$ 
\dim_{\mathbb{Q}} ( A  
\otimes_{\mathbb{Z} Q_m} \mathbb{Q}) \leq d. \dim_{\mathbb{Q}} ((\mathbb{Z} Q / I) \otimes_{\mathbb{Z} Q_m} \mathbb{Q})
$$
and it suffices to show that 
$$
\sup_{m}  \dim_{\mathbb{Q}} ((\mathbb{Z} Q / I) \otimes_{\mathbb{Z} Q_m} \mathbb{Q}) < \infty.$$
For every $m$ there is a natural number $\alpha_m$ such that $Q^{\alpha_m} \subseteq Q_m$,
and $\mathbb{Z} Q / I \otimes_{\mathbb{Z} Q_m} \mathbb{Q} $ is a quotient of $\mathbb{Z} Q / I \otimes_{\mathbb{Z} Q^{\alpha_m}} \mathbb{Q}$. Thus 
\begin{equation*} \label{eq023}
\dim_{\mathbb{Q}} ((\mathbb{Z} Q / I) \otimes_{\mathbb{Z} Q_m} \mathbb{Q}) \leq \dim_{\mathbb{Q}} ((\mathbb{Z} Q / I) \otimes_{\mathbb{Z} Q^{\alpha_m}} \mathbb{Q}),
\end{equation*}
and we have reduced to showing that 
$$
\sup_s \dim_{\mathbb{Q}} ((\mathbb{Z} Q / I) \otimes_{\mathbb{Z} Q^{s}} \mathbb{Q}) < \infty.
$$
The theorem now follows from Lemma \ref{l:tensor} and Theorem \ref{commutativealgebra}.
\end{proof}

\section{Characteristic $p$ case}

In this section we shall construct examples which show that the restriction to fields of characteristic $0$ in Theorem A is essential, even in the metabelian
case. To this end, we consider
the {\em ${\rm{mod }}  \ p$ virtual first  betti number} of a finitely generated group $G$,
$$
\vbp(G) = \sup \{ \dim H_1(S,\mathbb{F}_p) \mid S<G\text{ of finite index }\}.
$$ 

\begin{proposition} For every prime $p$ there exist finitely presented metabelian groups $\Gamma$ such that $\vbp(\Gamma)$ is infinite.
\end{proposition}

\begin{proof}
Let $Q$ be a free abelian group with generators $x$ and $y$ and 
let $A = \mathbb{F}_pQ/ I$, where $I$ is the ideal of $\mathbb{F}_pQ$ generated by $y - x^2 + x - 1$.
Then
$$
A \cong  \mathbb{F}_p[x, x^{-1}, \frac{1}{x^2-x+1}].
$$
Consider
$$
A_m = A \otimes_{\mathbb{Z} Q^{p^m}} \mathbb{F}_p \cong \mathbb{F}_p Q / (I, Q^{p^m} - 1).
$$
Since $(x^2 - x+1)^{p^m} - 1 = x^{2p^m} - x^{p^m} + 1 - 1 = x^{p^m} (x^{p^m} - 1)$, we have
$$
A_m = \mathbb{F}_p[x, x^{-1}, \frac{1}{x^2-x+1}] / (x^{p^m} - 1, (x^2 - x +1)^{p^m} - 1)
 = $$ $$\mathbb{F}_p[x, x^{-1}, \frac{1}{x^2-x+1}] / (x^{p^m} - 1)
 $$ is the localisation 
 $$
 B_m S^{-1}$$
 where $B_m = \mathbb{F}_p[x, x^{-1}] / (x^{p^m} - 1)$ and $S$ is the image of $\{ (x^2 - x + 1)^j \}_{j \geq 1}$ in $B_m$.
Note  that $x^{p^m} -1$ and $x^2 - x +1$ do not have a common root in any finite field extension of $\mathbb{F}_p$, for if $z$ 
were a common root we would have $1 =z^{2p^m} = (z - 1)^{p^m} = z^{p^m} -1 = 0$, which is a contradiction. Thus 
the polynomials $x^{p^m} - 1$ and $ (x^2 - x + 1)^j $ are coprime in $\mathbb{F}_p[x, x^{-1}]$ i.e. generate the whole ring as an ideal, 
and so the elements of $S$ are invertible in $B_m$. Therefore $B_m S^{-1} = B_m$ and
$$
\dim_{\mathbb{F}_p} A_m =  \dim_{\mathbb{F}_p} B_m S^{-1}  =  \dim_{\mathbb{F}_p} B_m = p^m.
$$ 
Now define 
$$
\Gamma = A \rtimes Q \hbox{ and } \Gamma_m = A \rtimes Q^{p^m}.
$$
Then, as  the split case of Lemma \ref{fnew1}(b) (with $\mathbb{F}_p$ coefficients in place of rational ones)
$$
\dim_{\mathbb{F}_p} H_1(\Gamma_m, \mathbb{F}_p) = \dim_{\mathbb{F}_p} A_m  + 2 = p^m + 2,
$$ which tends to infinity with $m$.

By the calculation of $\Sigma_A(Q)$ for $A = \mathbb{F}_p Q/
I$, where the ideal $I$ is 1-generated, \cite[Thm.~5.2]{BS2} or by the link between $\Sigma_A^c(Q)$  and valuation theory (as described
in subsection \ref{s:tensor}), we have that
$$
\Sigma_A^c(Q) = \{ [\chi_1], [\chi_2], [\chi_3] \}
$$
$$ \chi_1(x) = 0, \chi_1(y) = 1,
 \chi_2(x) = 1, \chi_2(y) = 0 \hbox{ and } \chi_3(x) = -1, \chi_3(y) = -2.$$
Thus $A$ is 2-tame as $\mathbb{Z} Q$-module, and by the classification of finitely presented metabelian groups in \cite{BieriStrebel}, $\Gamma$ is finitely presented.
\end{proof}

\begin{cor}
There exists a finitely presented metabelian group $G$ such
that for the class $\mathcal A$ of all subgroups of finite
index in $G$ 
$$
\sup_{M \in {\mathcal A}} d(M) = \infty,
$$
where $d(M)$ is the minimal number of generators of $M$.
\end{cor}

\begin{proof} Immediate, since
$$
d(M) \geq \dim_{\mathbb{F}_p} H_1(M, \mathbb{F}_p).
$$ 
\end{proof}

It is natural to wonder if the lack of finiteness exhibited in the preceding proposition might be avoided by restricting to
subgroups whose index is coprime to $p$. The following refinement shows that this is not the case. 

\begin{proposition}  Let $p$ be a prime.  There exist  finitely presented metabelian groups $G$ such that
$$
\sup \{ \dim_{\mathbb{F}_p} H_1(S,\mathbb{F}_p) \mid S \in {\mathcal{A}}_p \} = \infty,
$$
where
$${\mathcal{A}}_p = \{ S \leq G \mid [G : S] \hbox{ is finite and  coprime to } p \}.$$
\end{proposition}

\begin{proof} Let $A = \mathbb{F}_p[x, {\frac 1 x}, {\frac 1 {x+1}}]$ and
let $Q$ be a free abelian group of rank 2 whose generators $x_1,x_2$ act on $A$ as multiplication by
$x$ and $x+1$ respectively. We consider the group  $G = A \rtimes Q$. As an $\mathbb{F}_p[Q]$-module,
$A\cong \mathbb{F}_p[Q]/I$ where $I$ is the ideal generated by $x_2 - x_1-1$, and the
argument given in the preceding proposition shows that $\Sigma_A(Q)^c$ 
consists of precisely $3$ points, no pair of which is antipodal. Therefore
$G$ is finitely presented.

Let $F$ be a finite field with $p^r$ elements, $r \geq 2$. Let $w$ be a generator of the multiplicative group  $F^* = F \smallsetminus \{ 0 \}$.
Let $Q_r$ be the kernel of the homomorphism $Q\to F^*$ defined by $x_1\mapsto w$ and $x_2\mapsto w+1$. Let
$G_r = A\rtimes Q_r$ and note that $|G/G_r|=|Q/Q_r|=p^r-1$ is coprime to $p$.

The ring epimorphism $A\to F$ sending $x$ to $w$ provides an epimorphism of the underlying additive groups which extends to a group-epimorphism
$A\rtimes Q_r\to F \times \Z^2$. Since $\dim_{\mathbb{F}_p} F = r$, it follows that $\dim_{\mathbb{F}_p} H_1(G_r, \mathbb{F}_p) \ge r+2$.
And  $r \geq 2$ was arbitrary.
\end{proof}

\section{Beyond the metabelian case}

In this section we shall prove Theorem A, but first we present
a consequence of Theorem \ref{metabeliancase} that describes what one can deduce about towers of finite-index subgroups
above the commutator subgroup in amenable and related groups.

\begin{prop}
Let $G$ be a group of type $\FP_2$ that does not contain a non-abelian free group
and let $\mathcal{C}$ be the set of finite-index subgroups in $G$ that contain the
commutator subgroup $G'$. Then
 $\sup_{M\in\mathcal{C}} \dim_{\mathbb{Q}} H_1(M, \mathbb{Q}) < \infty$.
\end{prop}
\begin{proof} 
By \cite[Thm.~5.5]{BieriStrebel} $G/ G''$ is finitely presented.  Since $M \supseteq G'$ we have $M' \supseteq G''$, hence we can
replace  $G$ by $G/ G''$ and $M$ by $M G''/ G''$ without changing $H_1(M, \mathbb{Q})$. Then we can apply Theorem \ref{metabeliancase}.
\end{proof}

Our proof of Theorem A relies on the following proposition, which is of interest in its own right.

\begin{prop} \label{nilpotent} Let $N \into G \onto Q$ be a short exact sequence of groups, where $N$ is nilpotent, 
$Q$ is abelian and $G$ is finitely generated. Let $G_n$ be a
subgroup of finite index in $G$ and let $\overline{G}_n$ be the
image of $G_n$ in the metabelian group $G / N'$. Then
$$
\dim_{\mathbb{Q}} H_1(G_n, \mathbb{Q}) = \dim_{\mathbb{Q}} H_1(\overline{G}_n,
\mathbb{Q}).
$$
\end{prop}

\begin{proof}  We argue using the Malcev completion $j:N\to N^*$, as defined in \cite{Malcev}.
According to 
\cite[Appendix A, Cor.~3.8]{Quillen}, for any nilpotent group $N$ the
homomorphism $j_N: N \to {N}^*$ is characterized up to isomorphism by
the following properties:

(a) $N^*$ is nilpotent and uniquely divisible;

(b) $\ker j_N$ is the torsion subgroup of $N$;

(c) for every $ x \in {N}^*$ there is a positive integer $n$ such that $x^n \in N$.

In any nilpotent group, the set $\sqrt{S}$ of elements that have powers in a fixed subgroup $S$ is
a subgroup. It follows that, for every subgroup $M<N$, the map $M\to \sqrt{j_N(M)}$ satisfies
properties (a) to (c). Thus we may identify $M^*$ with $\sqrt{j_N(M)}<N^*$. If $M<N$ has finite index
then $M^*=\sqrt{j_N(M)}=N^*$. And $(N^*)' = (N')^*$.

With these facts in hand, for all subgroups of finite index $G_n<G$ we have 
$(G_n')^* \supseteq ((G_n \cap N)')^* = ((G_n \cap N)^*)' = (N^*)' = (N')^*$. Thus $(G_n' N')^* = (G_n')^*$, and from (c) we deduce that $G_n'N'/G_n'$ is torsion. As $G_n'N'/G_n'$ is the kernel of the canonical epimorphism $G_n / G_n' \to G_n N' / G_n' N'$, we have $H_1(G_n, \mathbb{Q}) \cong (G_n / G_n') \otimes \mathbb{Q} \cong (G_n N'/ G_n' N') \otimes \mathbb{Q} \cong  H_1(\overline{G}_n, \mathbb{Q})$ as required.
 \end{proof}
 
\begin{theorem}\label{nilpotent2} Let $N \into G \onto Q$ be a short exact sequence of groups. If $N$ is nilpotent, 
$Q$ is abelian and $G$ is of type $\FP_2$, then the virtual first betti number $\vb(G)$ is finite.
\end{theorem}

\begin{proof} In the light of the preceding proposition, this follows directly 
from Theorem \ref{metabeliancase} and the fact, proved in \cite[Thm.~5.5]{BieriStrebel}, that $G / N'$ is a finitely presented  metabelian group.
\end{proof}

\begin{cor}[=Theorem A]  \label{main}  If a group $G$ is nilpotent-by-abelian-by-finite group
and of type  $\FP_2$, then $\vb(G)$ is finite.
\end{cor}

\begin{proof}
Let $G_0$ be a subgroup of finite index in $G$ 
such that $G_0$ is nilpotent-by-abelian. Then $G_0$ has type $\FP_2$, so $\vb(G_0)$ is finite, by Theorem \ref{nilpotent2},
and hence so is $G$, by Lemma \ref{finiteindex}.
\end{proof}

\begin{remark}\label{r:weak-fp}
We did not use the full force of finite presentability in establishing Theorem A: in fact,
 it is enough to assume that $G$
has a subgroup of finite index $G_0$ in which there is a nilpotent subgroup $N\ns G_0$ such that $Q=G_0/ N $ is free abelian
and, writing   $A = N / N'$, the  $\mathbb{Q} Q$-module  $A \otimes A \otimes \mathbb{Q}$,
with diagonal action, should be finitely generated. These requirements follow from the finite presentability of $G_0 / N'$ but are strictly weaker.
\end{remark}

\begin{cor} Every soluble group of type $\FP_\infty$ has finite virtual first betti number.
\end{cor}

\begin{proof}
Soluble groups $S$ of type $\FP_{\infty}$ are constructible and hence nilpotent-by-abelian-by-finite 
\cite{constructible}. 
\end{proof}
 
 \begin{cor} Every abelian-by-polycyclic group of type $\FP_3$ has finite virtual first betti number.
\end{cor}

\begin{proof} By the main result of \cite{GKR}, abelian-by-polycyclic groups of type $\FP_3$ are
  nilpotent-by-abelian-by-finite.
\end{proof}

 \end{document}